\documentclass[11pt]{article}
\usepackage[T1]{fontenc}
\usepackage[utf8]{inputenc}
\usepackage{lmodern}
\usepackage{amsmath,amssymb,amsthm}
\usepackage[margin=1in]{geometry}
\usepackage{microtype}
\usepackage{graphicx}
\usepackage{url}
\usepackage{longtable}
\usepackage{array}
\newcommand{\E}{\mathbb E}
\newcommand{\R}{\mathbb R}
\newcommand{\tr}{\operatorname{tr}}
\newcommand{\qmin}{q_{\min}}
\newcommand{\ind}{\mathbf 1}
\newtheorem{lemma}{Lemma}
\newtheorem{theorem}{Theorem}
\usepackage[colorlinks=true,allcolors=blue]{hyperref}
\title{A sharp bound for the least signless Laplacian eigenvalue\\of triangle-free graphs}
\author{Dongxiu Cai\textsuperscript{1}%
\thanks{\textsuperscript{1} School of Mathematical Sciences, MOE-LSC,
SHL-MAC, Shanghai Jiao Tong University, 800 Dongchuan Road,
Shanghai 200240, P.R.~China. Emails:
\texttt{diudiutse@sjtu.edu.cn} (Dongxiu Cai) and
\texttt{czb911@sjtu.edu.cn} (Zhengbo Chen).}
\and Zhengbo Chen\textsuperscript{1}
\and Xiao-Dong Zhang\textsuperscript{1,2}%
\thanks{\textsuperscript{2} Corresponding author: Xiao-Dong Zhang.
Email: \texttt{xiaodong@sjtu.edu.cn}.}}
\date{}
\begin{document}
\begingroup
\makeatletter
\let\@fnsymbol\@gobble
\makeatother
\maketitle
\endgroup
\begin{abstract}
We prove that every finite simple triangle-free graph $G$ on $n\geq1$ vertices
satisfies $\qmin(G)\leq7n/50$, where $\qmin(G)$ is the least eigenvalue
of its signless Laplacian matrix $Q_G=D_G+A_G$, with $D_G$ the diagonal
matrix of vertex degrees and $A_G$ the adjacency matrix. The constant is sharp, since the Higman--Sims graph
and its balanced blow-ups attain equality, as shown by de Lima, Nikiforov
and Oliveira (2016). This proves the bound
conjectured by Balogh, Clemen, Lidick\'y, Norin and Volec
(2023). The proof uses five test column vectors associated
with an ordered edge to obtain a matrix Rayleigh inequality. Our use of
local Rayleigh inequalities together with flag-algebra semidefinite
programming is inspired by Zhou (2026). We use semidefinite
programming and rational approximation to find the required positive
semidefinite matrices, and complete the proof by exact correction and
verification.
\end{abstract}


\section{Introduction}\label{sec:introduction}

Let $G$ be a finite simple graph on $n\geq1$ vertices, with adjacency
matrix $A_G$. We write $\qmin(G)$ for the least eigenvalue of its signless
Laplacian
\[
Q_G=\operatorname{diag}_{v\in V(G)}(\deg_G(v))+A_G,
\]
where $\deg_G(v)=|\{w\in V(G):\{v,w\}\in E(G)\}|$ is the degree of $v$,
and $\operatorname{diag}_{z\in S}(b_z)$ denotes the diagonal matrix indexed
by a finite set $S$, with entry $b_z$ at $(z,z)$ and zero at every off-diagonal position.
For every real column vector $u$ indexed by $V(G)$,
\[
u^{\mathsf T}
Q_Gu
=\sum_{\{v,w\}\in E(G)}(u_v+u_w)^2.
\]
If $G$ is bipartite, choosing $u$ to be $1$ on one part and $-1$ on the
other gives a zero Rayleigh quotient, so $\qmin(G)=0$ by the nonnegativity
above.
Conversely, if $G$ is connected and $\qmin(G)=0$, a nonzero zero-eigenvector $u$
satisfies $u_v=-u_w$ on every edge; connectedness makes all coordinates
nonzero, and their signs give a bipartition.
Upper bounds on
$\qmin(G)$ for triangle-free graphs are related to questions about sparse
vertex sets and the number of edges that must be deleted to obtain a
bipartite graph~\cite{BCLNV,Brandt}.

Brandt~\cite{Brandt} conjectured that every regular triangle-free graph
of order $n$ satisfies $\qmin(G)\leq4n/25$. For a $d$-regular graph,
the degree matrix is $dI$, and therefore
\[
\qmin(G)=d+\lambda_{\min}(A_G)
=\lambda_{\max}(A_G)+\lambda_{\min}(A_G).
\]
Consequently, Brandt's conjecture can also be formulated as a bound on
the sum of the largest and smallest adjacency eigenvalues. Balogh,
Clemen, Lidick\'y, Norin and Volec~\cite{BCLNV} proved the stronger bound
\[
\qmin(G)\leq\frac{15}{94}n
\]
for every triangle-free graph. Their argument uses a test column vector
associated with an edge to bound the least eigenvalue by a ratio of
induced-path counts, and then bounds that ratio by flag algebras.

The Higman--Sims graph has $100$ vertices, is $22$-regular and
triangle-free, and has least adjacency eigenvalue $-8$. Its least
signless Laplacian eigenvalue is therefore $14$. In their concluding
remarks, Balogh et al.~\cite{BCLNV} propose that
$\qmin(G)\leq0.14n$ holds for every triangle-free graph. They also show
that their edge test column vector gives the larger coefficient $42/275$ on
the Higman--Sims graph. Improving only the subsequent path-count
inequality cannot reach the conjectured coefficient using that same
spectral reduction.

We prove the sharp bound conjectured by Balogh, Clemen,
Lidick\'y, Norin and Volec in 2023~\cite{BCLNV}.
\begin{theorem}\label{thm:bound}
For every finite simple triangle-free graph $G$ with $n=|V(G)|\geq1$,
\[
\qmin(G)\leq\frac7{50}n.
\]
\end{theorem}
The following known examples show that the bound in
Theorem~\ref{thm:bound} is sharp.
\begin{theorem}[{\cite{deLimaNikiforovOliveira}}]\label{thm:sharpness}
Let $K$ be the Higman--Sims graph. For every integer $t\geq1$, let $K[t]$
be obtained by replacing each vertex of $K$ by $t$ independent vertices
and each edge of $K$ by the complete bipartite graph between the corresponding
sets of copies. Then $K[t]$ is triangle-free, has $100t$ vertices, and satisfies
\[
\qmin(K[t])=14t=\frac7{50}|V(K[t])|.
\]
\end{theorem}
Hence $7/50$ is the smallest possible universal coefficient.
In particular, the same sharp coefficient bounds
$\lambda_{\max}(A_G)+\lambda_{\min}(A_G)$ for regular triangle-free graphs.

The spectral part of the proof uses five test column vectors for each ordered
edge. Combining their Rayleigh inequalities produces two real symmetric
$5\times5$ matrices: one records the averaged inner products of the test
columns, and the other records their averaged signless Laplacian forms
after subtracting the proposed bound. We construct a fixed positive
definite matrix whose trace product with the second matrix is
nonpositive. Its trace product with the first matrix is strictly positive
when the graph has an edge. The matrix Rayleigh inequality then yields
the desired eigenvalue bound. One of the test column vectors is obtained by
applying a quadratic polynomial in the normalized adjacency matrix to
the difference of two adjacency column vectors. This polynomial vanishes at
the nonprincipal normalized adjacency eigenvalues of the equality
example. The other column vectors also incorporate differences and local
variation of vertex degrees.

To prove the trace inequality, we expand the matrix entries into rational
linear combinations of graph homomorphism densities. Each term uses at
most eight sampled positions. We then compare coefficients on the
complete list of $410$ triangle-free graphs of order eight. The
comparison expresses the required trace as the negative of a sum of
nonnegative quadratic forms in conditional flag probabilities. These
quadratic forms use explicit integer matrices whose positivity is
verified by strict diagonal dominance. The remaining assertions are
$410$ exact finite identities. The accompanying data and programs
specify the integer entries and check these identities using exact
arithmetic.

The use of conditional flag probabilities follows the averaging principle
of flag algebras~\cite{Razborov}. Our use of local Rayleigh inequalities
together with flag-algebra semidefinite programming is inspired by Zhou's
work on $\{C_3,C_5\}$-free graphs. Zhou first averages local Rayleigh
inequalities~\cite[Lemma~2.1, pp.~4--5]{Zhou}, and then searches jointly
for positive semidefinite matrices for the local Rayleigh expressions and
ordinary flag squares, followed by rationalization and exact coefficient
verification~\cite[Section~2.3, p.~7]{Zhou}. This is the method we draw on;
the five test column vectors associated with an ordered edge and their
$5\times5$ matrix constraint are constructed in the present paper.
Here the eight positions are sampled
independently with replacement from the original graph. Their sampled
values may coincide, while the positions remain distinct vertices of
the sampled graph. The resulting identities are exact for every
$n\geq1$, so an eight-vertex calculation proves the bound for all graph
orders. We prove the required sampling identities directly.

Section~\ref{sec:preliminary} gives the graph and matrix conventions,
defines the homomorphism and flag densities used in the proof, and
establishes the finite sampling and positivity lemmas.
Section~\ref{sec:proof} constructs the five test column vectors, derives the
matrix Rayleigh inequality, specifies the finite integer identities,
and proves the bound.

The programs and exact data used in the proof are available at\newline
\url{https://github.com/Diudiutse/triangle-free-signless-laplacian}.

\section{Preliminary}\label{sec:preliminary}
\paragraph{Graphs and labels.}
A finite simple graph $B$ is a pair $(V(B),E(B))$, where $V(B)$ is a finite set and
\[
E(B)\subseteq\{\{u,v\}:u,v\in V(B),\ u\ne v\}.
\]
Edges are unordered two-element sets, so this convention excludes loops and multiple edges, and its adjacency matrix is
\[
(A_B)_{uv}=\begin{cases}1,&\{u,v\}\in E(B),\\0,&\{u,v\}\notin E(B),\end{cases}
\qquad (u,v\in V(B)).
\]
For $S\subseteq V(B)$, the induced subgraph is defined by
\[
B[S]=\left(S,\ E(B)\cap\{\{u,v\}:u,v\in S,\ u\ne v\}\right),
\]
so it retains exactly the edges of $B$ with both endpoints in $S$.
A graph $B$ is triangle-free precisely when
\[
\left\{\{u,v\},\{u,w\},\{v,w\}\right\}\nsubseteq E(B)
\quad\text{for all distinct }u,v,w\in V(B).
\]
For an integer $s\geq0$, write $[s]=\{1,\ldots,s\}$, with $[0]=\varnothing$; a graph on $[s]$ is a concrete graph whose vertex names are these integers, so two such graphs are equal exactly when their edge sets are equal.
For a map $f:S\to T$ and a subset $R\subseteq S$, write
$f(R)=\{f(r):r\in R\}$ for the image of $R$ and
$f|_R:R\to T$, $r\mapsto f(r)$, for the restriction of $f$ to $R$.
A graph isomorphism from $B$ to $B'$ is a bijection preserving both adjacency and nonadjacency, and the set of all such isomorphisms is
\begin{equation}\label{eq:isomorphisms}
\operatorname{Iso}(B,B')=
\left\{\phi:V(B)\overset{\mathrm{bij}}{\longrightarrow}V(B'):
(A_B)_{uv}=(A_{B'})_{\phi(u)\phi(v)}
\ \text{for all }u,v\in V(B)\right\}.
\end{equation}
We write $B\cong B'$ when $\operatorname{Iso}(B,B')\ne\varnothing$, and this is weaker than equality of graphs on $[s]$, because an isomorphism may permute all vertex names.
\paragraph{Graph maps.}
For finite simple graphs $J,B$, define the set of graph homomorphisms by
\begin{equation}\label{eq:hom-set}
\operatorname{Hom}(J,B)=
\left\{\phi:V(J)\longrightarrow V(B):
\{\phi(u),\phi(v)\}\in E(B)
\ \text{for every }\{u,v\}\in E(J)\right\}.
\end{equation}
A homomorphism may identify nonadjacent vertices of $J$, but it cannot identify the endpoints of an edge, because an edge cannot be sent to a loop in a simple graph.
The injective homomorphisms and the induced embeddings are, respectively,
\begin{align}
\operatorname{Emb}(J,B)
&=\{\phi\in\operatorname{Hom}(J,B):\phi\text{ is injective}\},\notag\\
\operatorname{Emb}_{\mathrm{ind}}(J,B)
&=\left\{\phi:V(J)\hookrightarrow V(B):
(A_J)_{uv}=(A_B)_{\phi(u)\phi(v)}
\ \text{for every }u,v\in V(J)\right\}.
\label{eq:embedding-sets}
\end{align}
In particular, an induced embedding preserves nonedges as well as edges, whereas an injective homomorphism need only preserve edges.
For $n\geq s\geq0$, write
\[
(n)_s=n(n-1)\cdots(n-s+1)=\frac{n!}{(n-s)!},
\qquad (n)_0=1,
\]
which is the number of injections from an $s$-element set into an $n$-element set.
\paragraph{Homomorphism densities.}
For a graph $B$ of order $n\geq1$ and a graph $J$ of order $s$, define
\begin{equation}\label{eq:hom-density}
t(J,B)=\frac{|\operatorname{Hom}(J,B)|}{n^s},
\qquad
t_{\mathrm{inj}}(J,B)=\frac{|\operatorname{Emb}(J,B)|}{(n)_s}
\quad(n\geq s).
\end{equation}
Writing $V(J)=[s]$ converts the first count to the explicit sum
\begin{equation}\label{eq:hom-product}
t(J,B)=\frac1{n^s}\sum_{Y_1,\ldots,Y_s\in V(B)}
\prod_{\{a,b\}\in E(J)}(A_B)_{Y_aY_b}.
\end{equation}
For a graph $J$ on $[s]$ and $n\geq s$, the ordinary induced-subgraph density is
\begin{equation}\label{eq:induced-density}
p(J,B)=
\frac{\left|\{S\subseteq V(B):|S|=s,\ B[S]\cong J\}\right|}
{\binom ns}.
\end{equation}
\begin{lemma}\label{lem:induced-normalization}
If $|V(J)|=s\leq n=|V(B)|$, then
\[
p(J,B)=
\frac{s!}{|\operatorname{Iso}(J,J)|}
\frac{|\operatorname{Emb}_{\mathrm{ind}}(J,B)|}{(n)_s}.
\]
\end{lemma}
\noindent\emph{Proof.}
For each $S\subseteq V(B)$ with $|S|=s$ and $B[S]\cong J$, choose
$\psi\in\operatorname{Iso}(J,B[S])$.
The map $\phi\mapsto\psi^{-1}\circ\phi$ is a bijection from
$\operatorname{Iso}(J,B[S])$ to $\operatorname{Iso}(J,J)$, so
\[
|\operatorname{Emb}_{\mathrm{ind}}(J,B)|
=|\operatorname{Iso}(J,J)|
\left|\{S\subseteq V(B):|S|=s,\ B[S]\cong J\}\right|.
\]
Dividing by $(n)_s=s!\binom ns$ gives the stated formula.\hfill$\square$

\paragraph{Sampling on positions.}
For a nonempty graph $B$ and a tuple $X=(X_1,\ldots,X_h)\in V(B)^h$, define the graph on the tuple positions by
\begin{equation}\label{eq:position-graph}
H_X=\left([h],\ 
\left\{\{a,b\}:1\leq a<b\leq h,\ \{X_a,X_b\}\in E(B)\right\}\right).
\end{equation}
For independent uniform vertex samples, the expectation and probability of any function or event are the explicit finite quantities
\[
\E_X f(X)=\frac1{|V(B)|^h}\sum_{X\in V(B)^h}f(X),
\qquad
\Pr(X\in\mathcal A)=\frac{|\mathcal A|}{|V(B)|^h}
\quad(\mathcal A\subseteq V(B)^h).
\]
Sampling with replacement includes all tuples in $V(B)^h$, whereas sampling without replacement restricts to tuples with pairwise distinct entries and divides by $(|V(B)|)_h$.
The ambient graph $B$ is fixed whenever $H_X$ is used, and for $n=|V(B)|\geq h$ the two sampling conventions are related by
\[
\Pr(H_X\cong J\mid X_1,\ldots,X_h\text{ are pairwise distinct})=p(J,B),
\]
because each $h$-element subset admits exactly $h!$ orderings.
Let $X$ and $Y$ both be uniformly distributed on $V(B)^h$.
For every permutation $\rho$ of $[h]$, the coordinate rearrangement $X\mapsto(X_{\rho(1)},\ldots,X_{\rho(h)})$ is a bijection of $V(B)^h$, and therefore
\begin{equation}\label{eq:exchangeability}
\E_X f(X_{\rho(1)},\ldots,X_{\rho(h)})
=\frac1{|V(B)|^h}\sum_{Y\in V(B)^h}f(Y)=\E_Yf(Y).
\end{equation}
\paragraph{Finite graph catalogs.}
The collection of all triangle-free graphs on the fixed set $[s]$ is
\begin{equation}\label{eq:graph-collection}
\mathcal G_s=
\left\{([s],E):
\begin{array}{l}
E\subseteq\{\{a,b\}:1\leq a<b\leq s\},\\
([s],E)\text{ is triangle-free}
\end{array}\right\}.
\end{equation}
For each $s=0,\ldots,8$, the finite catalog $\mathcal T_s\subseteq\mathcal G_s$ is required to satisfy
\begin{equation}\label{eq:catalog-uniqueness}
\left|\{H\in\mathcal T_s:B\cong H\}\right|=1
\qquad\text{for every }B\in\mathcal G_s,
\end{equation}
so the catalog contains one concrete edge set on $[s]$ from each ordinary isomorphism class.
The vertex names $1,\ldots,s$ have already been specified in \eqref{eq:graph-collection}; the choice in \eqref{eq:catalog-uniqueness} concerns which edge set represents a class, and does not choose a new set of vertex labels.
If the ambient graph $B$ is triangle-free, then $H_X$ is triangle-free: three pairwise adjacent positions would have three distinct values, since adjacent positions cannot have equal values, and those three values would form a triangle in $B$.
Thus, for $0\leq h\leq8$, independent $h$-tuples from any nonempty triangle-free $B$ give a partition
\begin{equation}\label{eq:source-partition}
V(B)^h=\bigsqcup_{H\in\mathcal T_h}\{X\in V(B)^h:H_X\cong H\},
\qquad
\sum_{H\in\mathcal T_h}\Pr(H_X\cong H)=1.
\end{equation}
The following finite sampling identity uses the same averaging idea as
\cite[proof of Proposition~11.32]{Lovasz}.
We give a direct proof using independent uniform samples from $V(B)$.
\begin{lemma}\label{lem:source-lift}
Let $B$ be a nonempty triangle-free graph, let $1\leq h\leq8$, and let $J$ be a simple graph of order $s\leq h$, including $s=0$.
For independent uniform $X_1,\ldots,X_h\in V(B)$,
\begin{equation}\label{eq:exact-source-lift}
t(J,B)=\E_X t_{\mathrm{inj}}(J,H_X)
=\sum_{H\in\mathcal T_h}t_{\mathrm{inj}}(J,H)\Pr(H_X\cong H).
\end{equation}
\end{lemma}
\noindent\emph{Proof.}
Fix an injection $\phi:V(J)\hookrightarrow[h]$; its images are different positions, so the values $(X_{\phi(a)})_{a\in V(J)}$ are independent uniform samples, although those values may coincide.
Using \eqref{eq:position-graph} and \eqref{eq:hom-product}, this gives
\[
\E_X\prod_{\{a,b\}\in E(J)}(A_{H_X})_{\phi(a)\phi(b)}
=\E_X\prod_{\{a,b\}\in E(J)}(A_B)_{X_{\phi(a)}X_{\phi(b)}}
=t(J,B).
\]
There are $(h)_s$ possible injections, so averaging their equal expectations gives
\begin{align*}
\E_X t_{\mathrm{inj}}(J,H_X)
&=\frac1{(h)_s}\sum_{\phi:V(J)\hookrightarrow[h]}
\E_X\prod_{\{a,b\}\in E(J)}(A_{H_X})_{\phi(a)\phi(b)}\\
&=\frac1{(h)_s}\,(h)_s\,t(J,B)=t(J,B).
\end{align*}
If $\eta:H\to H'$ is an isomorphism, then $\phi\mapsto\eta\circ\phi$ is a bijection from $\operatorname{Emb}(J,H)$ to $\operatorname{Emb}(J,H')$, so $t_{\mathrm{inj}}(J,H)=t_{\mathrm{inj}}(J,H')$.
This invariance makes $t_{\mathrm{inj}}(J,H_X)$ constant on each part of \eqref{eq:source-partition}, and summing the expectation over that partition proves the second equality in \eqref{eq:exact-source-lift}.\hfill$\square$

\paragraph{Types and ordered roots.}
A type of order $k$ is a specified triangle-free graph $\sigma=([k],E(\sigma))$, with $0\leq k\leq6$, and the proof uses each concrete $\sigma\in\mathcal T_k$ once.
A $\sigma$-rooted graph is a pair $(B,\theta)$, where $B$ is triangle-free and
\begin{equation}\label{eq:root-embedding}
\theta\in\operatorname{Emb}_{\mathrm{ind}}(\sigma,B),
\end{equation}
so the ordered roots are the vertices $\theta(1),\ldots,\theta(k)$ and their adjacency matrix is exactly $A_\sigma$.
Two $\sigma$-rooted graphs are isomorphic exactly when
\begin{equation}\label{eq:rooted-isomorphism}
(B,\theta)\cong_\sigma(B',\theta')
\quad\Longleftrightarrow\quad
\{\phi\in\operatorname{Iso}(B,B'):
\phi(\theta(i))=\theta'(i)\ (i\in[k])\}\ne\varnothing.
\end{equation}
A $\sigma$-flag is an isomorphism class under \eqref{eq:rooted-isomorphism}, and a flag of order $\ell$ has $\ell$ vertices, of which exactly $k$ are roots and $\ell-k$ are free vertices.
Every rooted graph $(B,\theta)$ of order $\ell$ can be represented on $[\ell]$: choose a bijection $\lambda:[\ell]\to V(B)$ with $\lambda(i)=\theta(i)$ for $i\in[k]$ and use
\[
\left([\ell],\{\{a,b\}:1\leq a<b\leq\ell,\ 
\{\lambda(a),\lambda(b)\}\in E(B)\}\right).
\]
If $\lambda'$ is another such bijection, then $(\lambda')^{-1}\circ\lambda$ fixes every $i\in[k]$ and is an isomorphism between the two graphs on $[\ell]$, so the resulting rooted isomorphism class does not depend on the ordering of the free vertices.
For $\ell\geq k$, the flag catalog $\mathcal F_{\sigma,\ell}$ consists of graphs satisfying
\begin{equation}\label{eq:flag-catalog}
\mathcal F_{\sigma,\ell}\subseteq
\{P\in\mathcal G_\ell:P[[k]]=\sigma\},
\qquad
\left|\{P\in\mathcal F_{\sigma,\ell}:(B,\iota)\cong_\sigma(P,\iota)\}\right|=1
\end{equation}
for every $B\in\mathcal G_\ell$ with $B[[k]]=\sigma$, where $\iota:[k]\hookrightarrow[\ell]$ is the inclusion $\iota(i)=i$.
For a catalog graph $P$, we also write $P$ for the rooted pair $(P,\iota)$, so $(B,\theta)\cong_\sigma P$ is always the explicit isomorphism condition in \eqref{eq:rooted-isomorphism}.
For the empty type $k=0$, there is a unique empty root map, \eqref{eq:rooted-isomorphism} reduces to ordinary graph isomorphism, and flag densities reduce to ordinary induced-subgraph densities.
\paragraph{Flag densities.}
If $(B,\theta)$ is $\sigma$-rooted, $|V(B)|=h$, and $P\in\mathcal F_{\sigma,\ell}$ with $k\leq\ell\leq h$, define
\begin{equation}\label{eq:flag-density}
p(P;(B,\theta))=
\frac{\left|\left\{S\subseteq V(B)\setminus\theta([k]):
|S|=\ell-k,\ 
(B[\theta([k])\cup S],\theta)\cong_\sigma P
\right\}\right|}
{\binom{h-k}{\ell-k}}.
\end{equation}
The induced graph in \eqref{eq:flag-density} retains every edge and nonedge among its roots and its free set, and $\theta$ remains the same ordered root map after restricting the graph.
For $P\in\mathcal F_{\sigma,\ell}$, $P'\in\mathcal F_{\sigma,\ell'}$, and $h\geq\ell+\ell'-k$, define the joint density by
\begin{equation}\label{eq:joint-flag-density}
p(P,P';(B,\theta))=
\frac{\displaystyle\left|\left\{(S,S'):
\begin{array}{l}
S,S'\subseteq V(B)\setminus\theta([k]),\quad S\cap S'=\varnothing,\\
|S|=\ell-k,\quad |S'|=\ell'-k,\\
(B[\theta([k])\cup S],\theta)\cong_\sigma P,\\
(B[\theta([k])\cup S'],\theta)\cong_\sigma P'
\end{array}\right\}\right|}
{\displaystyle\binom{h-k}{\ell-k}\binom{h-\ell}{\ell'-k}}.
\end{equation}
The denominator first chooses the free set for $P$ and then chooses the free set for $P'$ from the remaining vertices, so the numerator counts an ordered pair of disjoint free sets and tests no edges between those sets.
Every eligible free set, or pair of free sets, belongs to exactly one catalog flag, or pair of catalog flags, so
\[
\sum_{P\in\mathcal F_{\sigma,\ell}}p(P;(B,\theta))=1,
\qquad
\sum_{\substack{P\in\mathcal F_{\sigma,\ell}\\P'\in\mathcal F_{\sigma,\ell'}}}
p(P,P';(B,\theta))=1.
\]
\begin{lemma}[{\cite[Lemma~2.2]{Razborov}}]\label{lem:flag-chain}
Let $(B,\theta)$ be a $\sigma$-rooted graph of order $h$, and let $R$ range over $\mathcal F_{\sigma,d}$.
For $k\leq\ell\leq d\leq h$ and $P\in\mathcal F_{\sigma,\ell}$,
\[
p(P;(B,\theta))=\sum_Rp(P;R)\,p(R;(B,\theta)).
\]
If $P'\in\mathcal F_{\sigma,\ell'}$ and $\ell+\ell'-k\leq d\leq h$, then
\[
p(P,P';(B,\theta))
=\sum_Rp(P,P';R)\,p(R;(B,\theta)).
\]
\end{lemma}
The flag product is defined by these joint-density coefficients
\cite[equation~(5) and Lemma~2.4]{Razborov}: at any order $d\geq\ell+\ell'-k$, the product of $P,P'$ is represented by
\[
\sum_{R\in\mathcal F_{\sigma,d}}p(P,P';R)\,R,
\]
and Lemma~\ref{lem:flag-chain} shows how these coefficients behave when evaluated on a larger rooted graph.
For a formal linear combination $\sum_Pc_PP$ with all $P$ of order $\ell$, the coefficient of $R$ in its square is
\[
\sum_{P,P'\in\mathcal F_{\sigma,\ell}}c_Pc_{P'}\,p(P,P';R),
\]
obtained by distributing the two finite sums.
\paragraph{Integer ordering counts.}
For each type $\sigma$ of order $k$, we use flags of the order
\[
\ell=\left\lfloor\frac{8+k}{2}\right\rfloor,
\qquad
2\ell-k\leq8,
\qquad
\mathcal F_\sigma=\mathcal F_{\sigma,\ell},
\]
with $\ell$ specifying the number of vertices of a flag whenever this type is being considered.
For any triangle-free graph $H$ with eight vertices and $P,P'\in\mathcal F_\sigma$, define the integer matrix entry
\begin{equation}\label{eq:flag-count}
C_H^\sigma(P,P')=
\left|\left\{\rho:
\begin{array}{l}
\rho:[8]\overset{\mathrm{bij}}{\longrightarrow}V(H),\\
(A_H)_{\rho(i)\rho(j)}=(A_\sigma)_{ij}\quad(1\leq i<j\leq k),\\
(H[\rho([\ell])],\rho|_{[k]})\cong_\sigma P,\\
(H[\rho([k]\cup\{\ell+1,\ldots,2\ell-k\})],\rho|_{[k]})\cong_\sigma P'
\end{array}\right\}\right|.
\end{equation}
The two induced graphs in \eqref{eq:flag-count} use the same roots and the disjoint free-position sets
\[
\{k+1,\ldots,\ell\},
\qquad
\{\ell+1,\ldots,2\ell-k\}.
\]
Clearly, $C_H^\sigma$ is symmetric, that is,
$C_H^\sigma(P,P')=C_H^\sigma(P',P)$ for all $P,P'\in\mathcal F_\sigma$.
\begin{lemma}\label{lem:ordering-normalization}
For every triangle-free graph $H$ of order eight,
\begin{equation}\label{eq:ordering-product}
\frac{C_H^\sigma(P,P')}{8!}
=\frac1{(8)_k}
\sum_{\theta\in\operatorname{Emb}_{\mathrm{ind}}(\sigma,H)}
p(P,P';(H,\theta)).
\end{equation}
\end{lemma}
\noindent\emph{Proof.}
We express the root-averaging convention of \cite[Section~2.2]{Razborov}
in terms of the ordering count \eqref{eq:flag-count}.
Fix $\theta\in\operatorname{Emb}_{\mathrm{ind}}(\sigma,H)$ and an ordered
pair of disjoint sets $S,S'\subseteq V(H)\setminus\theta([k])$, each of
size $\ell-k$.
There are exactly
\[
(\ell-k)!^2(8-2\ell+k)!
\]
bijections $\rho:[8]\to V(H)$ satisfying
\[
\rho|_{[k]}=\theta,\qquad
\rho(\{k+1,\ldots,\ell\})=S,\qquad
\rho(\{\ell+1,\ldots,2\ell-k\})=S',
\]
because the two free groups and the unused positions can be ordered
independently.
Thus \eqref{eq:flag-count} and \eqref{eq:joint-flag-density} give
\begin{align*}
\frac{C_H^\sigma(P,P')}{8!}
&=\frac{(\ell-k)!^2(8-2\ell+k)!}{8!}
\binom{8-k}{\ell-k}\binom{8-\ell}{\ell-k}
\sum_{\theta\in\operatorname{Emb}_{\mathrm{ind}}(\sigma,H)}
p(P,P';(H,\theta))\\
&=\frac{(8-k)!}{8!}
\sum_{\theta\in\operatorname{Emb}_{\mathrm{ind}}(\sigma,H)}
p(P,P';(H,\theta))\\
&=\frac1{(8)_k}
\sum_{\theta\in\operatorname{Emb}_{\mathrm{ind}}(\sigma,H)}
p(P,P';(H,\theta)).
\end{align*}
This proves \eqref{eq:ordering-product}.\hfill$\square$

Summing \eqref{eq:ordering-product} over all $P,P'\in\mathcal F_\sigma$ gives the useful normalization
\[
\frac1{8!}\sum_{P,P'}C_H^\sigma(P,P')
=\frac{|\operatorname{Emb}_{\mathrm{ind}}(\sigma,H)|}{(8)_k}.
\]
\paragraph{Conditional flag probabilities.}
Now fix a nonempty triangle-free graph $B$ of order $n$, a tuple $r=(r_1,\ldots,r_k)\in V(B)^k$, and a flag $P\in\mathcal F_\sigma$, and let $H_r$ be the position graph defined by \eqref{eq:position-graph}.
The conditional probability used on $B$ is defined by
\begin{equation}\label{eq:conditional-probability}
\Pr(P\mid r)=
\begin{cases}
\displaystyle
\frac{\left|\left\{Y\in V(B)^{\ell-k}:
(H_{(r,Y)},\iota)\cong_\sigma P\right\}\right|}
{n^{\ell-k}},
& H_r=\sigma,\\[6pt]
0,& H_r\ne\sigma,
\end{cases}
\end{equation}
where $(r,Y)=(r_1,\ldots,r_k,Y_1,\ldots,Y_{\ell-k})$ and $\iota(i)=i$ maps the roots into the position graph.
For $H_r=\sigma$, the sets of tuples $Y$ producing the catalog flags partition $V(B)^{\ell-k}$, so
\[
\sum_{P\in\mathcal F_\sigma}\Pr(P\mid r)=1,
\]
while every summand is zero when $H_r\ne\sigma$.
For $k=0$, the root tuple is the unique empty tuple, $H_r$ is the empty type, and the root average consists of one term with probability one.
The joint event for a tuple of eight samples is the set
\begin{equation}\label{eq:joint-event}
\mathcal A^\sigma_{P,P'}=
\left\{X\in V(B)^8:
\begin{array}{l}
H_{(X_1,\ldots,X_k)}=\sigma,\\
(H_{(X_1,\ldots,X_\ell)},\iota)\cong_\sigma P,\\
(H_{(X_1,\ldots,X_k,X_{\ell+1},\ldots,X_{2\ell-k})},\iota)
\cong_\sigma P'
\end{array}\right\}.
\end{equation}
The nonnegativity of averaged flag squares is a standard result
\cite[Theorem~3.1(a)]{Razborov}.
The next lemma records the exact identity needed for the sampling rule
\eqref{eq:position-graph}, including tuples whose values in $V(B)$ repeat.
Its proof verifies this identity with the normalization $8!$ in
\eqref{eq:flag-count}.

\begin{lemma}\label{lem:exact-flag-square}
For every nonempty triangle-free $B$, independent uniform $X\in V(B)^8$, and $r=(X_1,\ldots,X_k)$,
\begin{equation}\label{eq:flag-entry-average}
\frac1{8!}\sum_{H\in\mathcal T_8}\Pr(H_X\cong H)C_H^\sigma(P,P')
=\E_r\left[\ind_{H_r=\sigma}\Pr(P\mid r)\Pr(P'\mid r)\right].
\end{equation}
Consequently, for arbitrary real coordinates $(w_P)_{P\in\mathcal F_\sigma}$,
\begin{align}
&\frac1{8!}\sum_{H\in\mathcal T_8}\Pr(H_X\cong H)
\sum_{P,P'}w_PC_H^\sigma(P,P')w_{P'}\notag\\
&\qquad=\E_r\left[\ind_{H_r=\sigma}
\left(\sum_Pw_P\Pr(P\mid r)\right)^2\right]\geq0.
\label{eq:flag-square}
\end{align}
In particular,
\begin{equation}\label{eq:averaged-count-psd}
\sum_{H\in\mathcal T_8}\Pr(H_X\cong H)C_H^\sigma
\quad\text{is positive semidefinite}.
\end{equation}
\end{lemma}
\noindent\emph{Proof.}
Write $n=|V(B)|$ and
$S_8=\{\rho:[8]\overset{\mathrm{bij}}{\longrightarrow}[8]\}$.
For $\rho\in S_8$, the two induced graphs in \eqref{eq:flag-count}
are identified with the position graphs in \eqref{eq:joint-event} by the maps
\[
\begin{array}{ll}
[\ell]\longrightarrow\rho([\ell]),&
i\longmapsto\rho(i),\\[3pt]
[\ell]\longrightarrow\rho([k]\cup\{\ell+1,\ldots,2\ell-k\}),&
i\longmapsto
\begin{cases}
\rho(i),&i\leq k,\\
\rho(\ell+i-k),&i>k.
\end{cases}
\end{array}
\]
Both maps send root $i$ to $\rho(i)$ and preserve adjacency by
\eqref{eq:position-graph}; hence
\[
C_{H_X}^\sigma(P,P')
=\left|\left\{\rho\in S_8:
(X_{\rho(1)},\ldots,X_{\rho(8)})\in\mathcal A^\sigma_{P,P'}\right\}\right|.
\]
For every $\rho\in S_8$, the maps
\[
(X_1,\ldots,X_8)\longmapsto(X_{\rho(1)},\ldots,X_{\rho(8)}),
\qquad
(Y_1,\ldots,Y_8)\longmapsto(Y_{\rho^{-1}(1)},\ldots,Y_{\rho^{-1}(8)})
\]
are inverse bijections of $V(B)^8$, so
\[
\left|\left\{X\in V(B)^8:
(X_{\rho(1)},\ldots,X_{\rho(8)})\in\mathcal A^\sigma_{P,P'}\right\}\right|
=|\mathcal A^\sigma_{P,P'}|.
\]
Using the partition \eqref{eq:source-partition} and the isomorphism
invariance of $C_H^\sigma$, we obtain
\begin{align*}
&\frac1{8!}\sum_{H\in\mathcal T_8}\Pr(H_X\cong H)C_H^\sigma(P,P')\\
&=\frac1{8!\,n^8}\sum_{H\in\mathcal T_8}
\sum_{\substack{X\in V(B)^8\\H_X\cong H}}C_H^\sigma(P,P')\\
&=\frac1{8!\,n^8}\sum_{X\in V(B)^8}C_{H_X}^\sigma(P,P')\\
&=\frac1{8!\,n^8}\sum_{X\in V(B)^8}\sum_{\rho\in S_8}
\ind_{\{(X_{\rho(1)},\ldots,X_{\rho(8)})\in\mathcal A^\sigma_{P,P'}\}}\\
&=\frac1{8!\,n^8}\sum_{\rho\in S_8}
\left|\left\{X\in V(B)^8:
(X_{\rho(1)},\ldots,X_{\rho(8)})\in\mathcal A^\sigma_{P,P'}\right\}\right|\\
&=\frac1{8!\,n^8}\sum_{\rho\in S_8}|\mathcal A^\sigma_{P,P'}|
=\frac{|\mathcal A^\sigma_{P,P'}|}{n^8}.
\end{align*}

For each fixed $r\in V(B)^k$, the definition \eqref{eq:joint-event}
gives the Cartesian product
\[
\begin{aligned}
&\left\{(Y,Z,W):
\begin{array}{l}
Y,Z\in V(B)^{\ell-k},\quad W\in V(B)^{8-2\ell+k},\\
(r,Y,Z,W)\in\mathcal A^\sigma_{P,P'}
\end{array}\right\}\\
&\quad=
\begin{cases}
\begin{aligned}
&\{Y\in V(B)^{\ell-k}:(H_{(r,Y)},\iota)\cong_\sigma P\}\\
&{}\times\{Z\in V(B)^{\ell-k}:(H_{(r,Z)},\iota)\cong_\sigma P'\}\\
&{}\times V(B)^{8-2\ell+k},
\end{aligned}
&H_r=\sigma,\\[6pt]
\varnothing,&H_r\ne\sigma.
\end{cases}
\end{aligned}
\]
Here $(r,Y,Z,W)$ is concatenation in the displayed order; its blocks
occupy positions $1,\ldots,k$, $k+1,\ldots,\ell$,
$\ell+1,\ldots,2\ell-k$, and $2\ell-k+1,\ldots,8$.
Taking cardinalities and applying \eqref{eq:conditional-probability} gives
\begin{align*}
\frac{|\mathcal A^\sigma_{P,P'}|}{n^8}
&=\frac1{n^8}\sum_{r\in V(B)^k}\ind_{H_r=\sigma}\\
&\qquad{}\cdot
\left|\{Y\in V(B)^{\ell-k}:(H_{(r,Y)},\iota)\cong_\sigma P\}\right|\\
&\qquad{}\cdot
\left|\{Z\in V(B)^{\ell-k}:(H_{(r,Z)},\iota)\cong_\sigma P'\}\right|
\,n^{8-2\ell+k}\\
&=\frac1{n^8}\sum_{r\in V(B)^k}\ind_{H_r=\sigma}
\,n^{\ell-k}\Pr(P\mid r)\,
n^{\ell-k}\Pr(P'\mid r)\,n^{8-2\ell+k}\\
&=\frac1{n^k}\sum_{r\in V(B)^k}
\ind_{H_r=\sigma}\Pr(P\mid r)\Pr(P'\mid r)\\
&=\E_r\left[\ind_{H_r=\sigma}\Pr(P\mid r)\Pr(P'\mid r)\right],
\end{align*}
where $(\ell-k)+(\ell-k)+(8-2\ell+k)=8-k$.
This proves \eqref{eq:flag-entry-average}.
Summing \eqref{eq:flag-entry-average} with coefficients $w_Pw_{P'}$
gives \eqref{eq:flag-square} by linearity of expectation.
This finishes the proof of Lemma~\ref{lem:exact-flag-square}.\hfill$\square$

\paragraph{Matrix positivity and traces.}
A real symmetric matrix $B$ is positive semidefinite when $u^{\mathsf T}Bu\geq0$ for every real column vector $u$, and positive definite when $u^{\mathsf T}Bu>0$ for every nonzero $u$; $\tr(B)=\sum_iB_{ii}$ denotes its trace.
For real matrices $B,C$ whose products are defined and square, direct summation gives
\[
\tr(BC)=\sum_{i,j}B_{ij}C_{ji}
=\sum_{j,i}C_{ji}B_{ij}=\tr(CB),
\]
and for symmetric $B,C$ this is also $\sum_{i,j}B_{ij}C_{ij}$.
\begin{lemma}[{\cite[Example~2.24]{BoydVandenberghe}}]\label{lem:trace-positivity}
If $B,C$ are real symmetric positive semidefinite matrices of the same size, then $\tr(BC)\geq0$.
\end{lemma}

\begin{lemma}\label{lem:integer-dominance}
Let $X$ be a symmetric integer matrix and let $q>0$ be an integer.
If every row satisfies
\[
q+X_{ii}-\sum_{j\ne i}|X_{ij}|>0,
\]
then $qI+X$ is positive definite.
\end{lemma}
\noindent\emph{Proof.}
The matrix $qI+X$ is real symmetric, so all its eigenvalues are real.
By Gershgorin's theorem~\cite[Theorem~1.1]{Varga}, for every eigenvalue $\lambda$ there is a row $i$ such that
\[
|\lambda-(q+X_{ii})|\leq\sum_{j\ne i}|X_{ij}|.
\]
Hence
\[
\lambda\geq q+X_{ii}-\sum_{j\ne i}|X_{ij}|>0,
\]
so $qI+X$ is positive definite.\hfill$\square$

For any real matrix $V$ and any positive semidefinite $B$, the product $VBV^{\mathsf T}$ is positive semidefinite because
\[
u^{\mathsf T}VBV^{\mathsf T}u=(V^{\mathsf T}u)^{\mathsf T}B(V^{\mathsf T}u)\geq0,
\]
and it is positive definite when $B$ is positive definite and $V$ is square and invertible.
\paragraph{The graph in the theorem.}
In the proof, $G=(V,E(G))$ is a finite simple triangle-free graph with $n=|V|\geq1$, $A=A_G$, and
\[
\deg_G(v)=|\{u\in V:\{u,v\}\in E(G)\}|
=\sum_{u\in V}A_{vu}.
\]
The least eigenvalue of the real symmetric matrix
\[
Q_G=\operatorname{diag}_{v\in V}(\deg_G(v))+A
\]
is denoted by $\qmin(G)$.
The adjacency column vector $A_{:x}$ is indexed by $V$ and has coordinate $(A_{:x})(z)=A_{zx}$, while for a real column vector $u$ indexed by $V$,
\[
\|u\|^2=u^{\mathsf T}u=\sum_{z\in V}u(z)^2,
\]
and $I$ denotes the identity matrix of the size required by its displayed product.
Subscripts $x,y,z,\ldots$ on $\E$ mean independent uniform samples from $V$ with replacement.
In particular, each ordered pair $(x,y)\in V^2$ has probability $1/n^2$.
Since $A_{xy}=1$ when $\{x,y\}\in E(G)$ and $A_{xy}=0$ otherwise, for every function $f:V^2\to\R$,
\[
\E_{x,y}[A_{xy}f(x,y)]
=\frac1{n^2}\sum_{(x,y)\in V^2}A_{xy}f(x,y)
=\frac1{n^2}\sum_{\substack{(x,y)\in V^2\\\{x,y\}\in E(G)}}f(x,y).
\]
Each edge $\{x,y\}$ contributes both ordered pairs $(x,y)$ and $(y,x)$, and each term still has weight $1/n^2$.

\section{Proof of the main theorem}\label{sec:proof}

To prove Theorem~\ref{thm:bound}, we first give a sufficient condition
in Lemma~\ref{lem:sufficient-condition}.
In order to state this lemma, we introduce the test column vectors
and the two averaged matrices used below.

For every ordered pair $(x,y)\in V^2$, define the first test column vector by
\begin{equation}\label{eq:U0}
U_0^{xy}=\left(\frac An-\frac1{50}I\right)(A_{:x}-A_{:y}),
\end{equation}
and, for every $z\in V$, matrix multiplication gives
\begin{align*}
U_0^{xy}(z)
&=\sum_{t\in V}
\left(\frac An-\frac1{50}I\right)_{zt}
(A_{:x}-A_{:y})(t)\\
&=\sum_{t\in V}
\left(\frac{A_{zt}}n-\frac{I_{zt}}{50}\right)(A_{tx}-A_{ty})\\
&=\frac1n\sum_{t\in V}A_{zt}(A_{tx}-A_{ty})
-\frac1{50}\sum_{t\in V}I_{zt}(A_{tx}-A_{ty})\\
&=\frac1n\sum_{t\in V}A_{zt}(A_{tx}-A_{ty})
-\frac{A_{zx}-A_{zy}}{50},
\end{align*}
where $(A_{:x}-A_{:y})(t)=A_{tx}-A_{ty}$ and
$I_{zt}=1$ for $t=z$ and $I_{zt}=0$ for $t\ne z$.
If $A_{xy}=1$, triangle-freeness gives $A_{xt}A_{yt}=0$ for every $t\in V$, because a common neighbor of the edge endpoints would form a triangle.
Using this vanishing product, symmetry, $A_{xx}=A_{yy}=0$, $A_{xy}=A_{yx}=1$, and $A_{zt}^2=A_{zt}$, the endpoint coordinates are
\begin{align}
U_0^{xy}(x)
&=\frac1n\sum_t(A_{xt}^2-A_{xt}A_{yt})-\frac{0-1}{50}
=\frac{\deg_G(x)}n+\frac1{50}>0,\notag\\
U_0^{xy}(y)
&=\frac1n\sum_t(A_{yt}A_{xt}-A_{yt}^2)-\frac{1-0}{50}
=-\frac{\deg_G(y)}n-\frac1{50}<0.
\label{eq:U0-endpoints}
\end{align}
In particular, \eqref{eq:U0-endpoints} proves that $U_0^{xy}\ne0$ for every ordered edge.
Define the next two column vectors directly from the adjacency column vectors and the degree diagonal entries by
\begin{align}
U_1^{xy}
&=\frac{\deg_G(x)-\deg_G(y)}n(A_{:x}+A_{:y}),\notag\\
U_2^{xy}
&=\operatorname{diag}_{z\in V}\left(
\frac{\deg_G(z)}n-\frac{\deg_G(x)+\deg_G(y)}{2n}\right)(A_{:x}-A_{:y}).
\label{eq:U12}
\end{align}
The fourth column vector is
\begin{equation}\label{eq:U3}
U_3^{xy}=\left(\frac{A^2}{n^2}+\frac{3A}{50n}-\frac1{625}I\right)(A_{:x}-A_{:y}).
\end{equation}
The numerical coefficients in \eqref{eq:U3} can also be read from the explicit multiplication
\begin{align*}
\left(\frac An+\frac2{25}I\right)U_0^{xy}
&=\left(\frac An+\frac2{25}I\right)
\left(\frac An-\frac1{50}I\right)(A_{:x}-A_{:y})\\
&=\left(\frac{A^2}{n^2}
+\left(\frac2{25}-\frac1{50}\right)\frac An
-\frac2{25\cdot50}I\right)(A_{:x}-A_{:y})\\
&=U_3^{xy}.
\end{align*}
The last column vector is
\begin{align}
U_4^{xy}
&=\operatorname{diag}_{z\in V}\left(
\frac{\deg_G(z)}n-\frac{\deg_G(x)+\deg_G(y)}{2n}\right)U_0^{xy}\notag\\
&\quad+\frac{(\deg_G(x)-\deg_G(y))^2}{4n^2}(A_{:x}-A_{:y}),
\label{eq:U4}
\end{align}
and $U^{xy}$ is the $n\times5$ matrix with these columns $U_0^{xy},\ldots,U_4^{xy}$ in order.

For $i,j=0,\ldots,4$, define
\begin{align}
\mathsf D_{ij}
&=\E_{x,y}\left[\frac{A_{xy}}n(U_i^{xy})^{\mathsf T}U_j^{xy}\right],\notag\\
\mathsf F_{ij}
&=\E_{x,y}\left[\frac{A_{xy}}n(U_i^{xy})^{\mathsf T}
\left(\frac{Q_G}{n}-\frac7{50}I\right)
U_j^{xy}\right].
\label{eq:DF}
\end{align}
Since $Q_G$ is real symmetric, transposing either quadratic product in \eqref{eq:DF} exchanges $i,j$ without changing its value, so both $\mathsf D,\mathsf F$ are real symmetric $5\times5$ matrices.
\begin{lemma}\label{lem:sufficient-condition}
If there exists a real symmetric positive definite $5\times5$ matrix $B$
such that, for every finite simple triangle-free graph $G$ with
$n=|V(G)|\geq1$,
\[
\tr(B\mathsf F)\leq0,
\]
where $\mathsf F$ is the matrix defined from that graph $G$ by \eqref{eq:DF},
then Theorem~\ref{thm:bound} holds.
Here the same matrix $B$ is used for every $G$.
\end{lemma}
\noindent\emph{Proof.}
Let $B$ be a matrix satisfying the hypotheses, and fix an arbitrary finite
simple triangle-free graph $G$ with $n=|V(G)|\geq1$.
If $E(G)=\varnothing$, then $Q_G=0$ and $\qmin(G)=0\leq7n/50$.
It therefore suffices to consider $E(G)\ne\varnothing$.
Choose an orthonormal eigenbasis $e_1,\ldots,e_n$ of $Q_G$, with corresponding eigenvalues $\lambda_1,\ldots,\lambda_n$, so every $\lambda_j\geq\qmin(G)$.
For every column vector $u\in\R^V$,
\[
u^{\mathsf T}\big(Q_G-\qmin(G)I\big)u
=\sum_{j=1}^n(\lambda_j-\qmin(G))(e_j^{\mathsf T}u)^2\geq0,
\]
and division by $n>0$ gives
\begin{equation}\label{eq:rayleigh}
u^{\mathsf T}\left(\frac{Q_G}{n}\right)u
\geq\frac{\qmin(G)}n\,u^{\mathsf T}u.
\end{equation}
Fix $u=(u_0,\ldots,u_4)^{\mathsf T}\in\R^5$.
The columns of $U^{xy}$ give
$U^{xy}u=\sum_{i=0}^4u_iU_i^{xy}$, so the first definition in
\eqref{eq:DF} yields the entrywise calculation
\begin{align*}
u^{\mathsf T}\mathsf Du
&=\sum_{i=0}^4\sum_{j=0}^4u_i\mathsf D_{ij}u_j\\
&=\sum_{i=0}^4\sum_{j=0}^4u_i u_j
\E_{x,y}\left[\frac{A_{xy}}n(U_i^{xy})^{\mathsf T}U_j^{xy}\right]\\
&=\E_{x,y}\left[\frac{A_{xy}}n
\sum_{i=0}^4\sum_{j=0}^4u_i u_j(U_i^{xy})^{\mathsf T}U_j^{xy}\right]\\
&=\E_{x,y}\left[\frac{A_{xy}}n
\left(\sum_{i=0}^4u_iU_i^{xy}\right)^{\mathsf T}
\left(\sum_{j=0}^4u_jU_j^{xy}\right)\right]\\
&=\E_{x,y}\left[\frac{A_{xy}}n
(U^{xy}u)^{\mathsf T}(U^{xy}u)\right]\\
&=\E_{x,y}\left[\frac{A_{xy}}n\|U^{xy}u\|^2\right]\\
&=\E_{x,y}\left[\frac{A_{xy}}n
\sum_{z\in V}\left(\sum_{i=0}^4u_iU_i^{xy}(z)\right)^2\right]\geq0.
\end{align*}
The last inequality uses $A_{xy}/n\geq0$ and the nonnegativity of every square;
hence $\mathsf D$ is positive semidefinite.
Its diagonal entries are averages of squared norms, and \eqref{eq:U0-endpoints} and $E(G)\ne\varnothing$ give
\begin{equation}\label{eq:positive-D}
\tr(\mathsf D)=\sum_{i=0}^4\E_{x,y}\left[\frac{A_{xy}}n\|U_i^{xy}\|^2\right]
\geq\E_{x,y}\left[\frac{A_{xy}}n\|U_0^{xy}\|^2\right]>0.
\end{equation}
For every $u=(u_0,\ldots,u_4)^{\mathsf T}\in\R^5$,
substitute both definitions in \eqref{eq:DF} and use
$U^{xy}u=\sum_{i=0}^4u_iU_i^{xy}$:
\begin{align}
&u^{\mathsf T}\left[\mathsf F-
\left(\frac{\qmin(G)}n-\frac7{50}\right)\mathsf D\right]u\notag\\
&=\sum_{i=0}^4\sum_{j=0}^4u_i u_j
\left[\mathsf F_{ij}-\left(\frac{\qmin(G)}n-\frac7{50}\right)\mathsf D_{ij}\right]\notag\\
&=\E_{x,y}\left[\frac{A_{xy}}n
\sum_{i=0}^4\sum_{j=0}^4u_i u_j
(U_i^{xy})^{\mathsf T}
\left(\frac{Q_G}{n}
-\frac7{50}I-\left(\frac{\qmin(G)}n-\frac7{50}\right)I\right)
U_j^{xy}\right]\notag\\
&=\E_{x,y}\left[\frac{A_{xy}}n
\left(\sum_{i=0}^4u_iU_i^{xy}\right)^{\mathsf T}
\left(\frac{Q_G}{n}
-\frac7{50}I-\left(\frac{\qmin(G)}n-\frac7{50}\right)I\right)
\left(\sum_{j=0}^4u_jU_j^{xy}\right)\right]\notag\\
&=\E_{x,y}\left[\frac{A_{xy}}n
(U^{xy}u)^{\mathsf T}
\left(\frac{Q_G}{n}
-\frac7{50}I-\frac{\qmin(G)}nI+\frac7{50}I\right)
(U^{xy}u)\right]\notag\\
&=\E_{x,y}\left[\frac{A_{xy}}n
(U^{xy}u)^{\mathsf T}
\left(\frac{Q_G-\qmin(G)I}{n}\right)
(U^{xy}u)\right]\geq0.
\label{eq:matrix-rayleigh}
\end{align}
The last inequality is pointwise before averaging: \eqref{eq:rayleigh}
applied to $U^{xy}u$ gives
\[
\begin{aligned}
&(U^{xy}u)^{\mathsf T}
\left(\frac{Q_G-\qmin(G)I}{n}\right)
(U^{xy}u)\\
&\quad=
(U^{xy}u)^{\mathsf T}
\left(\frac{Q_G}{n}\right)
(U^{xy}u)
-\frac{\qmin(G)}n\|U^{xy}u\|^2
\geq0,
\end{aligned}
\]
and $A_{xy}/n\geq0$ preserves this inequality when multiplying and averaging.
Because \eqref{eq:matrix-rayleigh} holds for every $u\in\R^5$, the symmetric matrix
$\mathsf F-(\qmin(G)/n-7/50)\mathsf D$ is positive semidefinite.

For the matrix $B$ in the hypotheses, linearity of the trace gives
\begin{align*}
&\tr(B\mathsf F)-
\left(\frac{\qmin(G)}n-\frac7{50}\right)\tr(B\mathsf D)\\
&=\tr\left(B\left[\mathsf F-
\left(\frac{\qmin(G)}n-\frac7{50}\right)\mathsf D\right]\right)\geq0.
\end{align*}
The last inequality follows from Lemma~\ref{lem:trace-positivity},
since both matrix factors are positive semidefinite.
Rearranging gives
\begin{equation}\label{eq:weighted-rayleigh}
\left(\frac{\qmin(G)}n-\frac7{50}\right)\tr(B\mathsf D)
\leq\tr(B\mathsf F).
\end{equation}
Since $B$ is positive definite, $B-\lambda_{\min}(B)I$ is positive semidefinite.
Since $\mathsf D$ is also positive semidefinite, Lemma~\ref{lem:trace-positivity} and linearity of the trace give
\[
\tr(B\mathsf D)-\lambda_{\min}(B)\tr(\mathsf D)
=\tr\bigl((B-\lambda_{\min}(B)I)\mathsf D\bigr)\geq0.
\]
Thus, by \eqref{eq:positive-D} and $\lambda_{\min}(B)>0$,
\[
\tr(B\mathsf D)\geq\lambda_{\min}(B)\tr(\mathsf D)>0.
\]
By hypothesis, $\tr(B\mathsf F)\leq0$.
Combining this with \eqref{eq:weighted-rayleigh} gives
\[
\left(\frac{\qmin(G)}n-\frac7{50}\right)\tr(B\mathsf D)
\leq\tr(B\mathsf F)\leq0.
\]
Dividing by the strictly positive quantity $\tr(B\mathsf D)$ gives
\[
\frac{\qmin(G)}n-\frac7{50}
\leq\frac{\tr(B\mathsf F)}{\tr(B\mathsf D)}\leq0,
\qquad\text{hence}\qquad
\qmin(G)\leq\frac7{50}n.
\]
Since $G$ was arbitrary, this proves Theorem~\ref{thm:bound}.\hfill$\square$

By Lemma~\ref{lem:sufficient-condition}, it remains to construct one real
symmetric positive definite $5\times5$ matrix $B$, independent of $G$,
for which $\tr(B\mathsf F)\leq0$ for every finite simple triangle-free graph
$G$ with at least one edge.

To calculate $\mathsf F$, we expand its entries into linear combinations
of homomorphism densities.
In the coordinate formulas below, write $U_i(z)=U_i^{xy}(z)$.
The variables listed in each expectation subscript are sampled independently
and uniformly from $V(G)$, with replacement.
In averages over $x,y$, suppress only the superscript $xy$ on each $U_i$, and first expand \eqref{eq:DF} as
\begin{align*}
\mathsf F_{ij}
&=\E_{x,y,z}\left[A_{xy}U_i(z)\frac{\deg_G(z)}nU_j(z)\right]\\
&\quad+\E_{x,y,z,t}\left[A_{xy}A_{zt}U_i(z)U_j(t)\right]
-\frac7{50}\E_{x,y,z}\left[A_{xy}U_i(z)U_j(z)\right].
\end{align*}
The three underlying coordinate sums are
\begin{align*}
(U_i^{xy})^{\mathsf T}\operatorname{diag}_{v\in V}(\deg_G(v))U_j^{xy}
&=\sum_z\deg_G(z)U_i(z)U_j(z),\\
(U_i^{xy})^{\mathsf T}AU_j^{xy}
&=\sum_{z,t}A_{zt}U_i(z)U_j(t),\\
(U_i^{xy})^{\mathsf T}U_j^{xy}
&=\sum_zU_i(z)U_j(z),
\end{align*}
and the factors $1/n$ in \eqref{eq:DF} give exactly the normalized averages in the first displayed expansion of $\mathsf F_{ij}$.
Substituting $\deg_G(z)/n=\E_tA_{zt}$ into the first term gives
\begin{align*}
\mathsf F_{ij}
&=\E_{x,y,z,t}[A_{xy}A_{zt}U_i(z)U_j(z)]\\
&\quad+\E_{x,y,z,t}[A_{xy}A_{zt}U_i(z)U_j(t)]
-\frac7{50}\E_{x,y,z}[A_{xy}U_i(z)U_j(z)].
\end{align*}
For fixed $x,y,z$, the first column vector has the expansion
\[
U_0(z)=\frac1n\sum_sA_{zs}(A_{sx}-A_{sy})-\frac1{50}(A_{zx}-A_{zy})
=\E_s[A_{zs}(A_{sx}-A_{sy})]-\frac1{50}(A_{zx}-A_{zy}).
\]
Replacing each degree in \eqref{eq:U12} by its adjacency sum gives
\begin{align*}
U_1(z)&=\E_s[(A_{xs}-A_{ys})(A_{zx}+A_{zy})],\\
U_2(z)&=\E_s\left[
\left(A_{zs}-\frac{A_{xs}+A_{ys}}2\right)(A_{zx}-A_{zy})\right].
\end{align*}
For fixed $x,y,z$, take the $z$-coordinate in \eqref{eq:U3} and use
\[
(A^2)_{zt}=\sum_{s\in V}A_{zs}A_{st},
\qquad
\sum_{t\in V}I_{zt}(A_{tx}-A_{ty})=A_{zx}-A_{zy}.
\]
Expanding the matrix entries, distributing the sum, and then normalizing gives
\begin{align*}
U_3(z)
&=\sum_{t\in V}
\left(\frac{A^2}{n^2}+\frac{3A}{50n}-\frac1{625}I\right)_{zt}
(A_{tx}-A_{ty})\\
&=\sum_{t\in V}
\left(\frac{(A^2)_{zt}}{n^2}+\frac{3A_{zt}}{50n}-\frac{I_{zt}}{625}\right)
(A_{tx}-A_{ty})\\
&=\frac1{n^2}\sum_{t\in V}(A^2)_{zt}(A_{tx}-A_{ty})
+\frac3{50n}\sum_{t\in V}A_{zt}(A_{tx}-A_{ty})\\
&\quad-\frac1{625}\sum_{t\in V}I_{zt}(A_{tx}-A_{ty})\\
&=\frac1{n^2}\sum_{t\in V}
\left(\sum_{s\in V}A_{zs}A_{st}\right)(A_{tx}-A_{ty})
+\frac3{50n}\sum_{t\in V}A_{zt}(A_{tx}-A_{ty})\\
&\quad-\frac1{625}(A_{zx}-A_{zy})\\
&=\frac1{n^2}\sum_{s,t\in V}A_{zs}A_{st}(A_{tx}-A_{ty})
+\frac3{50}\left(\frac1n\sum_{s\in V}A_{zs}(A_{sx}-A_{sy})\right)\\
&\quad-\frac1{625}(A_{zx}-A_{zy})\\
&=\E_{s,t}[A_{zs}A_{st}(A_{tx}-A_{ty})]
+\frac3{50}\E_s[A_{zs}(A_{sx}-A_{sy})]
-\frac1{625}(A_{zx}-A_{zy}).
\end{align*}
Replacing the normalized sums by expectations in \eqref{eq:U4} gives
\begin{align*}
U_4(z)
&=\E_s\left[A_{zs}-\frac{A_{xs}+A_{ys}}2\right]
\E_t\left[A_{zt}(A_{tx}-A_{ty})-\frac1{50}(A_{zx}-A_{zy})\right]\\
&\quad+\frac14(A_{zx}-A_{zy})
\E_s[A_{xs}-A_{ys}]\,\E_t[A_{xt}-A_{yt}].
\end{align*}
Combining these independent averages gives the fully expanded fifth column vector
\begin{align*}
U_4(z)
&=\E_{s,t}\left[
\left(A_{zs}-\frac{A_{xs}+A_{ys}}2\right)
\left(A_{zt}(A_{tx}-A_{ty})-\frac1{50}(A_{zx}-A_{zy})\right)\right]\\
&\quad+\frac14\E_{s,t}\left[
(A_{xs}-A_{ys})(A_{xt}-A_{yt})(A_{zx}-A_{zy})\right].
\end{align*}
Substitute the five expectation formulas above into the three terms of
$\mathsf F_{ij}$ and distribute the finite products.
By linearity of expectation, $\mathsf F_{ij}$ is a rational linear combination
of expectation monomials of the form
\[
\E_{v_1,\ldots,v_s}
\left[A_{v_{a_1}v_{b_1}}\cdots A_{v_{a_m}v_{b_m}}\right].
\]
Here $v_1,\ldots,v_s$ name the distinct positions occurring in the product;
all are sampled independently and uniformly, and their sampled values may
coincide. Inspecting the five column-vector formulas shows that each product obtained
from the first two terms of $\mathsf F_{ij}$ uses at most eight positions,
and each product obtained from its last term uses at most seven.
Since $A_{ab}\in\{0,1\}$, repeated factors can be replaced by one factor;
also $A_{ab}=A_{ba}$, so their order does not matter.
Terms containing $A_{aa}$ vanish because $A_{aa}=0$, and terms containing
$A_{ab}A_{bc}A_{ca}$ vanish because $G$ is triangle-free.
For each remaining term, form the simple graph $J$ on $[s]$ with edge set
\[
E(J)=\{\{a_r,b_r\}:1\leq r\leq m\}.
\]
The term is exactly its homomorphism density:
\[
\begin{aligned}
\E_{v_1,\ldots,v_s}
\left[\prod_{\{a,b\}\in E(J)}A_{v_av_b}\right]
&=\frac1{n^s}\sum_{(v_1,\ldots,v_s)\in V(G)^s}
\prod_{\{a,b\}\in E(J)}A_{v_av_b}\\
&=t(J,G).
\end{aligned}
\]
Thus, for $J\in\bigcup_{s=0}^8\mathcal T_s$, let $c_{ij,J}\in\mathbb Q$
be the sum of the rational coefficients of all terms whose graphs are
isomorphic to $J$, with $c_{ij,J}=0$ when there are no such terms. Then
\begin{equation}\label{eq:hom-expansion}
\mathsf F_{ij}=\sum_{J\in\bigcup_{s=0}^8\mathcal T_s}c_{ij,J}\,t(J,G).
\end{equation}
Exchanging $i,j$ commutes the factors in the degree and norm terms, while exchanging $z,t$ in the adjacency term and consistently relabeling its position variables gives a coefficient-preserving bijection of monomials, so $c_{ij,J}=c_{ji,J}$.

To express $\mathsf F$ as an average over $\mathcal T_8$, define, for each
$H\in\mathcal T_8$, the rational symmetric $5\times5$ matrix $F_H$ by
\begin{align}
(F_H)_{ij}
&=\sum_Jc_{ij,J}\,t_{\mathrm{inj}}(J,H)\notag\\
&=\sum_J\frac{c_{ij,J}}{(8)_{|V(J)|}}
\sum_{\phi:V(J)\hookrightarrow V(H)}
\prod_{\{a,b\}\in E(J)}(A_H)_{\phi(a)\phi(b)}.
\label{eq:FH}
\end{align}
Applying Lemma~\ref{lem:source-lift} to every summand of \eqref{eq:hom-expansion} and interchanging finite sums gives
\begin{align}
\mathsf F_{ij}
&=\sum_Jc_{ij,J}\sum_{H\in\mathcal T_8}
t_{\mathrm{inj}}(J,H)\Pr(H_X\cong H)\notag\\
&=\sum_{H\in\mathcal T_8}(F_H)_{ij}\Pr(H_X\cong H),
\qquad
\mathsf F=\sum_{H\in\mathcal T_8}F_H\Pr(H_X\cong H),
\label{eq:F-source-lift}
\end{align}
where $X$ is an independent uniform eight-tuple from $V$.

By Lemma~\ref{lem:sufficient-condition}, it remains to construct
one real symmetric positive definite $5\times5$ matrix $B$, independent of $G$,
and prove
\[
\tr(B\mathsf F)\leq0
\qquad\text{for every finite triangle-free graph }G\text{ with at least one edge}.
\]
Here $\mathsf F$ is the matrix defined from $G$ in \eqref{eq:DF}.
By \eqref{eq:F-source-lift} and the linearity of the trace,
\[
\tr(B\mathsf F)
=\sum_{H\in\mathcal T_8}\Pr(H_X\cong H)\tr(BF_H).
\]
Thus the only dependence on $G$ in this expression is through the probabilities
$\Pr(H_X\cong H)$.
The matrices $F_H$ are the fixed rational matrices in \eqref{eq:FH}.

By Lemma~\ref{lem:exact-flag-square}, for each type $\sigma$, the matrix
\[
\sum_{H\in\mathcal T_8}\Pr(H_X\cong H)C_H^\sigma
\quad\text{is positive semidefinite}.
\]
For each type $\sigma\in\bigcup_{k=0}^6\mathcal T_k$, we will choose a real
symmetric positive semidefinite matrix $M_\sigma$, independent of $G$,
whose rows and columns are indexed by $\mathcal F_\sigma$.
We seek to express $\tr(B\mathsf F)$ in the following form:
\[
\tr(B\mathsf F)
=-\sum_\sigma\tr\left(M_\sigma
\sum_{H\in\mathcal T_8}\Pr(H_X\cong H)C_H^\sigma\right).
\]
All sums over $\sigma$ here and below range over
$\bigcup_{k=0}^6\mathcal T_k$.
The following lemma makes this reduction to the eight-vertex identities precise.
\begin{lemma}\label{lem:finite-trace-reduction}
If there exist a real symmetric positive definite $5\times5$ matrix $B$
and, for every type $\sigma\in\bigcup_{k=0}^6\mathcal T_k$, a real symmetric
positive semidefinite matrix $M_\sigma$ whose rows and columns are indexed
by $\mathcal F_\sigma$, such that
\[
\tr(BF_H)
=-\sum_\sigma\tr(M_\sigma C_H^\sigma)
\qquad\text{for every }H\in\mathcal T_8,
\]
then
\[
\tr(B\mathsf F)\leq0
\]
for every finite simple triangle-free graph $G$ with $|V(G)|\geq1$,
where $\mathsf F$ is defined from $G$ by \eqref{eq:DF}.
Consequently, Theorem~\ref{thm:bound} holds.
\end{lemma}
\noindent\emph{Proof.}
Fix an arbitrary finite simple triangle-free graph $G$ with $|V(G)|\geq1$,
and let $X$ be an independent uniform eight-tuple from $V(G)$.
Using \eqref{eq:F-source-lift}, the assumed identities, and linearity of the trace,
we obtain
\begin{align*}
\tr(B\mathsf F)
&=\sum_{H\in\mathcal T_8}\Pr(H_X\cong H)\tr(BF_H)\\
&=-\sum_{H\in\mathcal T_8}\Pr(H_X\cong H)
\sum_\sigma\tr(M_\sigma C_H^\sigma)\\
&=-\sum_\sigma\tr\left(
M_\sigma
\sum_{H\in\mathcal T_8}\Pr(H_X\cong H)C_H^\sigma\right)
\leq0.
\end{align*}
The third equality interchanges the finite sums and uses linearity of the trace.
For each $\sigma$, the first matrix factor in the last line is positive
semidefinite by hypothesis, and the second is positive semidefinite by
Lemma~\ref{lem:exact-flag-square}.
Thus their trace product is nonnegative by Lemma~\ref{lem:trace-positivity},
which gives the last inequality.
Since $G$ was arbitrary and $B$ is positive definite,
Lemma~\ref{lem:sufficient-condition} proves Theorem~\ref{thm:bound}.
This finishes the proof.\hfill$\square$

To find the matrices in Lemma~\ref{lem:finite-trace-reduction}, we used
semidefinite programming with the symmetric matrices $B,M_\sigma$ and a
real scalar $\tau$ as unknowns, while $F_H$ and $C_H^\sigma$ were fixed:
\[
\begin{aligned}
\text{minimize}\quad &\tau\\
\text{subject to}\quad
&B_{00}=1,\qquad B-10^{-8}I\succeq0,\qquad \tau\geq0,\\
&M_\sigma\succeq0
\quad\left(\sigma\in\bigcup_{k=0}^6\mathcal T_k\right),\\
&-\tau\leq10^6\left(\tr(BF_H)
+\sum_\sigma\tr(M_\sigma C_H^\sigma)\right)\leq\tau
\quad(H\in\mathcal T_8).
\end{aligned}
\]
Here $S\succeq0$ means that $S$ is positive semidefinite, and the sizes of
$B$ and $M_\sigma$ are those in Lemma~\ref{lem:finite-trace-reduction}.
The normalization $B_{00}=1$ excludes the zero solution, and the lower
bound on $B$ ensures positive definiteness.
The objective minimizes the maximum absolute error in the $410$ identities,
scaled by $10^6$; a feasible solution with $\tau=0$ satisfies those identities exactly.
The constant $7/50$ in $F_H$ is fixed throughout this search.
We then rationalized the numerical matrix factors and corrected the remaining
errors by solving linear equations over the rationals, obtaining the exact
integer data below, whose positivity and identities are checked in
Lemma~\ref{lem:finite-certificate}.

To specify explicitly the matrices obtained from the SDP after rational
correction, we introduce the following integer data.
To construct a positive definite $B$, take a positive integer $q$, an invertible
integer $5\times5$ matrix $L$, and a symmetric integer $5\times5$ matrix $X_*$,
and set
\[
B=L(qI+X_*)L^{\mathsf T}.
\]
We will choose them so that $qI+X_*$ is positive definite; then $B$ is positive definite since $L$ is invertible.

For each $\sigma\in\bigcup_{k=0}^6\mathcal T_k$, to construct a positive
semidefinite $M_\sigma$, take a rectangular integer matrix $V_\sigma$
whose rows are indexed by $P\in\mathcal F_\sigma$, and a
symmetric integer matrix $X_\sigma$ whose size equals the number of columns
of $V_\sigma$.
We will require $qI+X_\sigma$ to be positive definite as well.
Then the square matrix
\[
V_\sigma(qI+X_\sigma)V_\sigma^{\mathsf T}
\]
has rows and columns indexed by the same flags as $C_H^\sigma$, and is positive semidefinite.
For these constructions, Lemma~\ref{lem:integer-dominance} gives a direct
way to check the required positive definiteness: it suffices to verify
\[
q+(X_\tau)_{ii}>\sum_{j\ne i}|(X_\tau)_{ij}|
\quad\text{for every row }i,
\qquad \tau\in\{*\}\cup\bigcup_{k=0}^6\mathcal T_k.
\]
The trace of each such product means the explicit finite sum
\[
\tr\left(V_\sigma(qI+X_\sigma)V_\sigma^{\mathsf T}C_H^\sigma\right)
=\sum_{P,P'\in\mathcal F_\sigma}
[V_\sigma(qI+X_\sigma)V_\sigma^{\mathsf T}]_{PP'}C_H^\sigma(P,P').
\]

For positive integers $a_\sigma$, we choose
\[
M_\sigma=\frac{2^{80}}{8!\,a_\sigma^2}
V_\sigma(qI+X_\sigma)V_\sigma^{\mathsf T}.
\]
This matrix is positive semidefinite because $2^{80}/(8!a_\sigma^2)>0$
and $V_\sigma(qI+X_\sigma)V_\sigma^{\mathsf T}$ is positive semidefinite.
Thus the fixed identities to verify become
\[
\tr(BF_H)
=-\sum_\sigma\frac{2^{80}}{8!\,a_\sigma^2}
\tr\left(V_\sigma(qI+X_\sigma)V_\sigma^{\mathsf T}C_H^\sigma\right)
\qquad(H\in\mathcal T_8).
\]
The coefficients $2^{80}/a_\sigma^2$ are fixed positive rational numbers
in the representation; their numerical values will be specified below.
By Lemma~\ref{lem:finite-trace-reduction}, these identities and the stated
positivity suffice to prove Theorem~\ref{thm:bound}.
It remains to specify $L,q,X_*,X_\sigma,V_\sigma$ explicitly and check the following conditions:
\begin{enumerate}
\renewcommand{\theenumi}{\roman{enumi}}
\renewcommand{\labelenumi}{(\theenumi)}
\item\label{cond:L-invertible} The integer matrix $L$ is invertible.
\item\label{cond:spectral-positive} The matrix $qI+X_*$ is positive definite. By Lemma~\ref{lem:integer-dominance}, it suffices to check
\[
q+(X_*)_{ii}-\sum_{j\ne i}|(X_*)_{ij}|>0
\qquad(i=0,\ldots,4).
\]
\item\label{cond:flag-positive} For every type $\sigma$, the matrix $qI+X_\sigma$ is positive definite. By Lemma~\ref{lem:integer-dominance}, it suffices to check
\[
q+(X_\sigma)_{ii}-\sum_{j\ne i}|(X_\sigma)_{ij}|>0
\qquad\text{in every row }i.
\]
\item\label{cond:finite-identity} For a fixed positive integer $\Delta$, the following identity holds for every $H\in\mathcal T_8$:
\begin{align*}
0&=\frac{\Delta}{2^{80}}\,
\tr\left(L(qI+X_*)L^{\mathsf T}F_H\right)\\
&\quad+\sum_\sigma\frac{\Delta}{40320a_\sigma^2}
\tr\left(V_\sigma(qI+X_\sigma)V_\sigma^{\mathsf T}C_H^\sigma\right).
\end{align*}
\end{enumerate}

We now briefly introduce how we compute the required coefficients and matrices,
specify the fixed integer data, and verify these conditions.\footnote{\raggedright The programs, data and manuscript are in the GitHub repository
\href{https://github.com/Diudiutse/triangle-free-signless-laplacian}{\path{Diudiutse/triangle-free-signless-laplacian}}.
All file paths below are relative to its root.}

\paragraph{Step 1. Computing $c_{ij,J}$.}
The program \href{https://github.com/Diudiutse/triangle-free-signless-laplacian/blob/main/programs/01_expand_F.py}{\path{programs/01_expand_F.py}} implements the five expectation formulas
for $U_0(z),\ldots,U_4(z)$ displayed above and substitutes them into
$\mathsf F_{ij}$.
Each term is stored as its rational coefficient and a list of adjacency
factors, with a distinct name for each distinct summation position.
In the complete expansion of $\mathsf F_{ij}$, every such position is
averaged independently and uniformly; this convention is not stored as
additional data.
Addition, scalar multiplication and multiplication first expand the
entire expression by the distributive law.
The program then removes terms containing a loop or triangle, removes
repeated adjacency factors, and adds the coefficients of isomorphic graphs.
The resulting coefficients are precisely those in \eqref{eq:hom-expansion}.
The output \href{https://github.com/Diudiutse/triangle-free-signless-laplacian/blob/main/programs/results/01/c_ij_J.pkl}{\path{programs/results/01/c_ij_J.pkl}} is a dictionary with keys $(i,j)$,
$0\leq i,j\leq4$; its value at $(i,j)$ is the dictionary
$\{J\mapsto c_{ij,J}\}$.
A graph is stored as a pair \texttt{(vertices, edges)} of tuples:
the first tuple lists its vertices, and the second lists its edges as pairs
$(a,b)$ with $a<b$, representing the unordered edges $\{a,b\}$.
The stored vertex set is $[s]$ for a graph of order $s$.
All coefficients are exact rational numbers, stored using
\texttt{Fraction}.

\paragraph{Step 2. Generating $\mathcal T_8$.}
The program \href{https://github.com/Diudiutse/triangle-free-signless-laplacian/blob/main/programs/02_generate_tau8.py}{\path{programs/02_generate_tau8.py}} starts with the empty graph on no
vertices and adds one vertex at a time.
For every retained graph on $[s]$, it considers every independent set
$S\subseteq[s]$ and adds vertex $s+1$ with neighborhood exactly $S$.
It then retains one representative of each ordinary isomorphism class,
using NetworkX for the isomorphism test.
The new graph is triangle-free because any new triangle would contain
$s+1$ and two adjacent vertices in $S$.
Conversely, deleting a vertex from any triangle-free graph leaves a
triangle-free graph in which that vertex's former neighborhood is independent.
Thus induction on the number of vertices proves that the enumeration is complete.
The numbers of representatives at orders $0,\ldots,8$ are
\[
1,1,2,3,7,14,38,107,410.
\]
The output \href{https://github.com/Diudiutse/triangle-free-signless-laplacian/blob/main/programs/results/02/tau_8.pkl}{\path{programs/results/02/tau_8.pkl}} is the list of the $410$ representatives
$H\in\mathcal T_8$, in the generated order.

\paragraph{Step 3. Computing $F_H$.}
The program \href{https://github.com/Diudiutse/triangle-free-signless-laplacian/blob/main/programs/03_compute_F_H.py}{\path{programs/03_compute_F_H.py}} loads the coefficient dictionary from
\href{https://github.com/Diudiutse/triangle-free-signless-laplacian/blob/main/programs/results/01/c_ij_J.pkl}{\path{programs/results/01/c_ij_J.pkl}} and the graph list from
\href{https://github.com/Diudiutse/triangle-free-signless-laplacian/blob/main/programs/results/02/tau_8.pkl}{\path{programs/results/02/tau_8.pkl}}.
Its function \texttt{tinj(J,H)} enumerates all injections
$\phi:V(J)\hookrightarrow V(H)$, counts those preserving every edge of $J$,
and divides by the total number of injections.
It therefore returns $t_{\mathrm{inj}}(J,H)$ as an exact rational number.
For each $H$ and each $(i,j)$, the program evaluates exactly the sum
in \eqref{eq:FH}:
\[
(F_H)_{ij}=\sum_Jc_{ij,J}t_{\mathrm{inj}}(J,H).
\]
The output \href{https://github.com/Diudiutse/triangle-free-signless-laplacian/blob/main/programs/results/03/F_H.pkl}{\path{programs/results/03/F_H.pkl}} is a dictionary whose key is $H$ and
whose value is the dictionary $\{(i,j)\mapsto(F_H)_{ij}\}$.
It contains $410$ matrices, each with $25$ rational entries.

\paragraph{Step 4. Generating the types $\sigma$.}
The program \href{https://github.com/Diudiutse/triangle-free-signless-laplacian/blob/main/programs/04_generate_sigma.py}{\path{programs/04_generate_sigma.py}} calls the graph-generation
function from \href{https://github.com/Diudiutse/triangle-free-signless-laplacian/blob/main/programs/02_generate_tau8.py}{\path{programs/02_generate_tau8.py}} for each order $k=0,\ldots,6$.
It lists the resulting representatives in increasing order of $k$ and
saves them in \href{https://github.com/Diudiutse/triangle-free-signless-laplacian/blob/main/programs/results/04/sigma_list.pkl}{\path{programs/results/04/sigma_list.pkl}}.
There are $1+1+2+3+7+14+38=66$ types.
For a retained representative on $[k]$, these vertex numbers become the
fixed root labels in the subsequent flag computation.

\paragraph{Step 5. Generating $\mathcal F_\sigma$.}
The program \href{https://github.com/Diudiutse/triangle-free-signless-laplacian/blob/main/programs/05_generate_flags.py}{\path{programs/05_generate_flags.py}} reads the type list from step~4.
For each $\sigma$ of order $k$, it starts with $\sigma$ and successively
adds vertices $k+1,\ldots,\ell$, where $\ell=\lfloor(8+k)/2\rfloor$,
again considering every independent neighborhood.
At each order, it identifies two graphs only when an isomorphism fixes
$1,\ldots,k$ individually.
The function \texttt{labeled\_graphs\_are\_isomorphic} tests this condition
by enumerating all permutations of the remaining vertices and checking
whether one of them preserves every adjacency and nonadjacency.
Deleting an unlabeled vertex gives the same induction as in step~2,
now with all root labels fixed, so every $\sigma$-flag of order $\ell$ is represented.
The output \href{https://github.com/Diudiutse/triangle-free-signless-laplacian/blob/main/programs/results/05/F_sigma.pkl}{\path{programs/results/05/F_sigma.pkl}} is the dictionary
$\{\sigma\mapsto\text{the ordered list of flags in }\mathcal F_\sigma\}$.
This list fixes the row and column order of $C_H^\sigma$ and the row
order of $V_\sigma$.

\paragraph{Step 6. Computing $C_H^\sigma$.}
The program \href{https://github.com/Diudiutse/triangle-free-signless-laplacian/blob/main/programs/06_compute_C.py}{\path{programs/06_compute_C.py}} reads the graph list from step~2 and
the flag lists from step~5.
For each $(\sigma,H)$, the function \texttt{compute\_C\_matrix} computes the
entire matrix in \eqref{eq:flag-count}.
It enumerates all ordered choices of $k$ distinct roots in $H$ inducing
exactly the labeled type $\sigma$.
For each such root map $\theta$, it enumerates pairs of subsets satisfying
\[
S,T\subseteq V(H)\setminus\theta([k]),\qquad
|S|=|T|=\ell-k,\qquad S\cap T=\varnothing.
\]
It finds the flag representatives $P,P'$ for which
\[
(H[\theta([k])\cup S],\theta)\cong_\sigma P,
\qquad
(H[\theta([k])\cup T],\theta)\cong_\sigma P',
\]
and adds
\[
((\ell-k)!)^2(8-2\ell+k)!
\]
to the entry indexed by $(P,P')$.
The two factors $(\ell-k)!$ count the orderings within $S$ and $T$,
and the last factorial counts the orderings of the unused vertices.
The root order has already been fixed, so this counts every bijection in
\eqref{eq:flag-count} exactly once.
The output \href{https://github.com/Diudiutse/triangle-free-signless-laplacian/blob/main/programs/results/06/C.pkl}{\path{programs/results/06/C.pkl}} is the dictionary
$\{(\sigma,H)\mapsto C_H^\sigma\}$, containing $66\cdot410=27060$ complete
integer matrices in the flag orders from step~5.

\paragraph{Step 7. Specifying the fixed integers.}
After these six computations, the remaining inputs are the fixed integers
$L,q,\Delta,X_*$ and $V_\sigma,X_\sigma,a_\sigma$.
The file \href{https://github.com/Diudiutse/triangle-free-signless-laplacian/blob/main/programs/spectral_data.py}{\path{programs/spectral_data.py}} supplies $L,q,\Delta,X_*$, with
$q,\Delta>0$, and
\[
\resizebox{\textwidth}{!}{$
L=\begin{pmatrix}
1099511627776&0&0&0&0\\
-232335715041&58226488264&0&0&0\\
646804735047&36701051459&62858792171&0&0\\
-3529332685417&-399541705998&9914953481&414703530043&0\\
-47136593763&-33884256159&53733927090&-22938839431&195039391412
\end{pmatrix}.
$}
\]
It is lower triangular with positive diagonal entries, so
\[
\det L=
1099511627776\cdot58226488264\cdot62858792171
\cdot414703530043\cdot195039391412>0.
\]
This verifies condition~(\ref{cond:L-invertible}).

The file \href{https://github.com/Diudiutse/triangle-free-signless-laplacian/blob/main/programs/sigma_data.py}{\path{programs/sigma_data.py}} supplies a dictionary \texttt{sigma\_data}.
Its key is a type $\sigma$ from step~4; the value is a dictionary with
exactly three entries:
\texttt{V\_sigma} stores $V_\sigma$, \texttt{X\_sigma} stores $X_\sigma$,
and \texttt{a\_sigma} stores the positive integer $a_\sigma$.
The rows of $V_\sigma$ follow the flag list from step~5, and the indices
of $X_\sigma$ follow the columns of $V_\sigma$.
The symmetric matrices $X_*$ and $X_\sigma$ are specified in these files
by upper-triangular triples $(i,j,c)$, with $i\leq j$.
For $\tau=*$ or a type $\sigma$, these triples determine the full matrix by
\begin{equation}\label{eq:correction-entries}
(X_\tau)_{ij}=(X_\tau)_{ji}=
\begin{cases}
c,&\text{if the triple }(i,j,c)\text{ is listed for }X_\tau,\\
0,&\text{if no triple is listed for }X_\tau\text{ at }(i,j),
\end{cases}
\qquad(i\leq j).
\end{equation}
Matrix indices in the programs start at zero.
Importing the two files produces the complete integer matrices with these entries.

\paragraph{Step 8. Verifying the remaining conditions.}
The last two programs check conditions~(\ref{cond:spectral-positive}),
(\ref{cond:flag-positive}) and~(\ref{cond:finite-identity}) using these inputs.
Their exact conclusions are as follows.
\begin{lemma}\label{lem:finite-certificate}
For the fixed integer data just specified, every
$\tau\in\{*\}\cup\bigcup_{k=0}^6\mathcal T_k$ satisfies
\begin{equation}\label{eq:certificate-row-margin}
q+(X_\tau)_{ii}-\sum_{j\ne i}|(X_\tau)_{ij}|>0
\qquad\text{in every row }i,
\end{equation}
and every $H\in\mathcal T_8$ satisfies
\begin{align}
0&=\frac{\Delta}{2^{80}}\,
\tr\left(L(qI+X_*)L^{\mathsf T}F_H\right)\notag\\
&\quad+\sum_\sigma\frac{\Delta}{40320a_\sigma^2}
\tr\left(V_\sigma(qI+X_\sigma)V_\sigma^{\mathsf T}C_H^\sigma\right).
\label{eq:integer-certificate}
\end{align}
\end{lemma}
\noindent\emph{Proof by exact finite verification.}
First, \href{https://github.com/Diudiutse/triangle-free-signless-laplacian/blob/main/programs/verify_39.py}{\path{programs/verify_39.py}} imports $q,X_*$ from \href{https://github.com/Diudiutse/triangle-free-signless-laplacian/blob/main/programs/spectral_data.py}{\path{programs/spectral_data.py}}
and the dictionary \texttt{sigma\_data} from \href{https://github.com/Diudiutse/triangle-free-signless-laplacian/blob/main/programs/sigma_data.py}{\path{programs/sigma_data.py}}.
Its function \texttt{satisfies\_39(X)} checks
\[
q+X_{ii}>\sum_{j\ne i}|X_{ij}|
\]
in every row of the input matrix $X$.
The program applies this function first to $X_*$, and then to each
$X_\sigma$ in the dictionary.
It stops and reports the corresponding matrix if any inequality fails.
All $1650$ rows of the $67$ matrices pass this check, proving
\eqref{eq:certificate-row-margin}.
Lemma~\ref{lem:integer-dominance} then verifies
conditions~(\ref{cond:spectral-positive}) and~(\ref{cond:flag-positive}).

Next, \href{https://github.com/Diudiutse/triangle-free-signless-laplacian/blob/main/programs/verify_40.py}{\path{programs/verify_40.py}} imports $L,q,\Delta,X_*$ from
\href{https://github.com/Diudiutse/triangle-free-signless-laplacian/blob/main/programs/spectral_data.py}{\path{programs/spectral_data.py}} and all $V_\sigma,X_\sigma,a_\sigma$ from
\href{https://github.com/Diudiutse/triangle-free-signless-laplacian/blob/main/programs/sigma_data.py}{\path{programs/sigma_data.py}}.
It reads the four outputs
\[
\begin{gathered}
\text{\href{https://github.com/Diudiutse/triangle-free-signless-laplacian/blob/main/programs/results/02/tau_8.pkl}{\path{programs/results/02/tau_8.pkl}}},\qquad
\text{\href{https://github.com/Diudiutse/triangle-free-signless-laplacian/blob/main/programs/results/03/F_H.pkl}{\path{programs/results/03/F_H.pkl}}},\\
\text{\href{https://github.com/Diudiutse/triangle-free-signless-laplacian/blob/main/programs/results/05/F_sigma.pkl}{\path{programs/results/05/F_sigma.pkl}}},\qquad
\text{\href{https://github.com/Diudiutse/triangle-free-signless-laplacian/blob/main/programs/results/06/C.pkl}{\path{programs/results/06/C.pkl}}}.
\end{gathered}
\]
Before checking the identity, it checks that the dictionary keys cover
exactly the required graphs and types, that all $C_H^\sigma$ have the sizes
determined by the flag lists, that the row counts of $V_\sigma$ agree
with those lists, and that $a_\sigma>0$ and
$\Delta/(40320a_\sigma^2)\in\mathbb Z$ for every type.

It forms $B=L(qI+X_*)L^{\mathsf T}$ once, and then loops over every
$H$ in the graph list from step~2.
For that $H$, it reads $F_H$ and the matching matrices $C_H^\sigma$
and evaluates the entire left-hand side of \eqref{eq:integer-certificate}.
The first trace is computed as $\tr(BF_H)$.
For each term in the sum over $\sigma$, the program uses
\begin{align}
\tr\left(V_\sigma(qI+X_\sigma)V_\sigma^{\mathsf T}C_H^\sigma\right)
&=\tr\left((qI+X_\sigma)(V_\sigma^{\mathsf T}C_H^\sigma V_\sigma)\right)
\notag\\
&=\sum_{i,j}(qI+X_\sigma)_{ji}
(V_\sigma^{\mathsf T}C_H^\sigma V_\sigma)_{ij}.
\label{eq:flag-trace-evaluation}
\end{align}
The first equality uses cyclicity of the trace, and the second is its
entrywise definition.
The code implements the last sum by transposing the first matrix,
multiplying corresponding entries of the two matrices, and summing all entries.
All matrix operations use NumPy arrays with \texttt{dtype=object}, so
the entries remain Python integers or exact rational \texttt{Fraction} values.
The program stops if the value for any $H$ is nonzero.
All $410$ values are exactly zero, proving \eqref{eq:integer-certificate}
and condition~(\ref{cond:finite-identity}).
The output \href{https://github.com/Diudiutse/triangle-free-signless-laplacian/blob/main/programs/results/verify_40.json}{\path{programs/results/verify_40.json}} records these $410$ zero values.
This finishes the verification of Lemma~\ref{lem:finite-certificate}.\hfill$\square$

\noindent\emph{Proof of Theorem~\ref{thm:bound}:}
The result follows from Lemmas~\ref{lem:sufficient-condition},
\ref{lem:finite-trace-reduction} and~\ref{lem:finite-certificate}.\hfill$\square$


\begin{samepage}
\section*{Statement on AI use}
The authors used GPT-6.1 Sol in the preparation of this work,
including the development of proof arguments, the checking of
intermediate steps, the generation and execution of semidefinite
programming (SDP) code and code for verifying inequalities and equalities,
and the drafting and revision of the manuscript and its
LaTeX source. All mathematical results, arguments, and proofs were
independently reviewed and verified by the authors, who take full
responsibility for the accuracy and content of this work.\par
\end{samepage}

\end{document}